\documentclass[english, 12pt]{article}
\usepackage[T1]{fontenc}
\usepackage{geometry}
\usepackage[latin9]{inputenc}
\usepackage{amsthm}
\usepackage{amsmath}
\usepackage{amssymb}
\usepackage{authblk}
\usepackage[margin=2cm, labelfont = bf]{caption}
\usepackage{setspace}
\usepackage{esint}
\usepackage{enumerate}
\usepackage{etoolbox}
\usepackage{url}
\usepackage{graphicx, epstopdf}
\usepackage[english]{babel}
\usepackage{xcolor}
\usepackage{soul}
\usepackage{tikz}
\usetikzlibrary{automata,arrows,positioning,calc}

\numberwithin{equation}{section}

\makeatletter
\theoremstyle{plain}
\newtheorem{thm}{\protect\theoremname}[section]
\ifx\proof\undefined
\newenvironment{proof}[1][\protect\proofname]{\par
\normalfont\topsep6\p@\@plus6\p@\relax
\trivlist
\itemindent\parindent
\item[\hskip\labelsep
\scshape
#1]\ignorespaces
}{%
\endtrivlist\@endpefalse
}
\providecommand{\proofname}{Proof}
\fi
\theoremstyle{plain}
\newtheorem{lem}[thm]{\protect\lemmaname}
\theoremstyle{plain}
\newtheorem{prop}[thm]{\protect\propositionname}
\theoremstyle{plain}

\theoremstyle{definition}

\theoremstyle{remark}

\theoremstyle{plain}
\newtheorem{cor}[thm]{\protect\corollaryname}
\theoremstyle{definition}

\numberwithin{figure}{section}

\makeatother

\usepackage{babel}
\providecommand{\conjecturename}{Conjecture}
\providecommand{\corollaryname}{Corollary}
\providecommand{\definitionname}{Definition}
\providecommand{\examplename}{Example}
\providecommand{\lemmaname}{Lemma}
\providecommand{\propositionname}{Proposition}
\providecommand{\remarkname}{Remark}
\providecommand{\theoremname}{Theorem}

\begin{document}


\title{Uniqueness for contagious McKean--Vlasov systems in the weak feedback regime}

\author[ ]{Sean Ledger}
\author[1]{Andreas S{\o}jmark}
\affil[1]{Department of Mathematics, Imperial College London}
\date{29 November, 2018}

\maketitle

\begin{abstract}
We present a simple uniqueness argument for a collection of McKean--Vlasov problems that have seen recent interest. Our first result shows that, in the weak feedback regime, there is global uniqueness for a very general class of random drivers. By weak feedback we mean the case where the contagion parameters are small enough to prevent blow-ups in solutions.  Next, we specialise to a Brownian driver and show how the same techniques can be extended to give short-time uniqueness after blow-ups, regardless of the feedback strength. The heart of our approach is a surprisingly simple probabilistic comparison argument that is robust in the sense that it does not ask for any regularity of the solutions.
\end{abstract}


\section{Introduction}

In this paper we study uniqueness of the McKean--Vlasov problem 
\begin{align}
\label{MV}
\tag{MV}
\begin{cases}
 X_t = X_0 + Z_t - \alpha f( L_t ) \\
\tau = \inf\{ t \geq 0 : X_t \leq 0 \}  \\
L_t = \mathbb{P}(\tau \leq t),
\end{cases}
\end{align}
where $X_0 \in (0,\infty)$ is a random start point, $Z$ is a continuous stochastic process, $f:[0,1]\rightarrow\mathbb{R}$ is a continuous function, and  $\alpha \in \mathbb{R}$ is a constant. By a solution to this problem we mean an increasing c\`adl\`ag function $L : [0,\infty) \to [0,1]$ that satisfies (\ref{MV}) and is initially zero.

We are mainly interested in the cases where $\alpha>0$ and the function $f$ is non-negative and increasing. This corresponds to a \emph{contagious} system with \emph{positive feedback}: as the particle $X_t$ gets closer to the absorbing boundary at zero, its probability of being absorbed increases, so $f(L_t)$ increases, and the particle is thus pushed even closer to the absorbing boundary. From precisely this perspective, variants of (\ref{MV}) have been studied in \cite{hambly_ledger_sojmark_2018,hambly_sojmark_2018,ledger_sojmark_2018,nadtochiy_shkolnikov_2017,nadtochiy_shkolnikov_2018} motivated by the study of contagion in large financial markets, and in \cite{carillo_et_al_2015,dirt_annalsAP_2015,dirt_SPA_2015,inglis_talay_2015} motivated by nonlinear excitatory integrate-and-fire models for large networks of electrically coupled neurons. In view of these applications, a particularly relevant setting is that of $f(x)=x$ and $\alpha>0$ together with $Z_t= \sqrt{1-\rho^2} B_t + \rho \beta_t$, where $B_t$ is a Brownian motion and $\beta_t$ is a fixed Brownian path. This problem corresponds to a pathwise realisation of the `conditional' (and contagious) McKean--Vlasov system
\begin{align}
\label{CMV}
\tag{CMV}
\begin{cases}
X_t = X_0 + \sqrt{1-\rho^2} B_t +  \rho B^0_t  - \alpha L_t \\
\tau = \inf\{ t \geq 0 : X_t \leq 0 \}  \\
L_t = \mathbb{P}(\tau \leq t \mid B^0 ),
\end{cases}
\end{align}
 where $(B,B^0)$ is a 2d Brownian motion independent of $X_0$. In Section 2 we demonstrate the global well-posedness of this system under a smallness condition on $\alpha>0$ defining the \emph{weak feedback regime}.
 
 In \cite{ledger_sojmark_2018} it is shown that `relaxed' solutions to (\ref{CMV})---for which the adaptedness of $L$ to $B^0$ is relaxed in a suitable way---arise as limit points of the following finite particle system:~$N$ particles move according to Brownian motions correlated through $B^0$ (known as the `common noise'), except that when a particle hits the origin, it is absorbed, and this then has a contagious effect causing all the other particles jump down by $\alpha/N$ with $\alpha>0$, possibly leading to more particles being absorbed, and hence further rounds of downward jumps.
 
 If, as in \cite{ledger_sojmark_2018}, each particle measures the `distance-to-default' of a financial entity, and absorption at zero corresponds to default, then such a positive feedback loop could model a \emph{cascade} of bankruptcies caused by the interplay between default contagion ($\alpha>0$) and common exposures ($\rho>0)$. Similarly, motivated by \cite{cont_eric_2018}, each particle could model the log leverage ratio of a bank (defined as the log of capital over assets) for which a minimum value is enforced by regulation. When reaching this minimum, the banks must sell assets to increase their leverage ratios, and if these sales pertain to common illiquid assets ($\rho>0$), then it depresses the price of these, hence causing a drop in the leverage ratios of the other banks ($\alpha>0$). Note, however, that a bank reaching the threshold should now be reset to some higher leverage ratio (after the selling of assets) instead of defaulting, but the main mathematical difficulty is still the positive feedback from hitting a boundary. Also, it would be natural for $\rho$ to depend on $L$ so that the selling of the common asset makes the banks less correlated.
 
By a simple sign change, this latter system can be rephrased as a model for the `spiking' of electrically coupled neurons, which has been studied in \cite{dirt_annalsAP_2015, dirt_SPA_2015}. In this case, each particle models the membrane potential of a neuron, and when this potential reaches an upper threshold the neuron is said to `spike': that is, it emits an electrical signal which causes all the other potentials to increase by $\alpha/N$ and the spiking neuron itself is reset to a predetermined value.
 
\subsection{Blow-ups and the physical jump condition}\label{subsec:PJC}
As a result of the positive feedback, for large enough $\alpha>0$, the solutions to (\ref{CMV}) can develop jump discontinuities which we call \emph{blow-ups}, as in \cite{hambly_ledger_sojmark_2018, ledger_sojmark_2018}. In order to study uniqueness in these cases, it is necessary to resolve ambiguity at the blow-ups (see \cite[Prop.~1.2 \& Sect.~2]{hambly_ledger_sojmark_2018} for illustrations and further discussion). This is achieved by specifying that the appertaining jump sizes must satisfy the physical jump condition
\begin{equation}
\label{PJC}
\tag{PJC}
\Delta L_t = \inf\{ x > 0 : \nu_{t-}( \hspace{0.2pt} [0,\alpha x ]\hspace{0.2pt}  ) < x \},
\end{equation}
with probability $1$, where $\nu_{t-}$ is the left limit of $\nu_t$ defined by
\[
\nu_t (S) := \mathbb{P}(X_t \in S, \, t < \tau \mid B^0) \quad \text{for all } S \in \mathcal{B}(\mathbb{R}),
\]
with $\tau=\inf \{ t\geq0 : X_t \leq 0 \}$ and $\nu_0$ is the law of $X_0$. That is, the flow $\nu$ gives the marginal laws of $X$ absorbed at the origin conditional on the common noise $B^0$. 
We note that $L$ is deterministic in (\ref{MV}), so naturally the condition (\ref{PJC}) should then be understood in terms of $\nu_t=\mathbb{P}(X_t \in \cdot\,, \, t < \tau)$ just as for (\ref{CMV}) with $\rho=0$.

In short, (\ref{PJC}) is the correct specification of the jump sizes for two reasons. First, in the case of (\ref{CMV}), the condition gives the minimal jump sizes that are necessary for $L$ to be c\`adl\`ag \cite[Prop.~3.3]{ledger_sojmark_2018}. Secondly, the solutions constructed in \cite{dirt_SPA_2015, ledger_sojmark_2018} are obtained as limit points of finite particle systems (as described above), for which the corresponding discrete version of (\ref{PJC}) gives the only sensible physical description of the jump behaviour, and this discrete condition then yields (\ref{PJC}) in the limit. It should also be noted that  generic (possibly non-physical) solutions to (\ref{MV}) can be constructed directly from a generalised Schauder fixed point theorem as proved by Nadtochyi and Shkolnikov \cite{nadtochiy_shkolnikov_2018} under suitable conditions on $Z$ and $f$.

With the physical jump condition in place, it is then natural to only consider initial densities that do not want to jump immediately (in alignment with $L$ being c\`adl\`ag and initially zero). Indeed, it is natural to restrict solely  to states that could be reached by the evolving system and, recalling that (\ref{PJC}) is a rule for the left limit $\nu_{t-}$, this translates as those initial conditions $\nu_0$ satisfying
\begin{equation}\label{intro_restrict_intial}
\inf\{ x > 0 : \nu_{0}( \hspace{0.2pt} [0,\alpha x ]\hspace{0.2pt}  ) < x \} = 0
\end{equation}
in the case $f(x)=x$ (for general $f$, the left-hand side in \eqref{intro_restrict_intial} should be replaced by the right-hand side in \eqref{pjc_general_f} with $t=0$ and $L_{0}=0$). While this condition is necessary for c\`adl\`ag solutions under Brownian drivers as in (\ref{CMV}) by \cite[Prop.~3.3]{ledger_sojmark_2018}, it is possible to construct c\`adl\`ag solutions to (\ref{MV}) violating this condition for more general $Z$. Specifically, in analogy with \cite[(3.15)]{nadtochiy_shkolnikov_2018}, one can take $Z_t:=B_t+A_t$ with $A_t:=\alpha \mathbb{P}(X_0+\inf_{s<t}B_s\leq 0)$ and $f(x)=x$. If $X_0$ is such that the left-hand side of (\ref{intro_restrict_intial}) is strictly positive, and it is \emph{not} imposed that this positivity should cause $L$ to jump immediately (as \eqref{PJC} would imply\label{key}), then $L:=A$ gives a continuous solution to (\ref{MV}) despite violating (\ref{intro_restrict_intial}) by the choice of $X_0$.

In this paper, the only mathematical contributions concerning jumps will be for the case when $Z$ is a standard Brownian motion and $f(x) = x$ (see Section 3). However, if considering a generic feedback function $f$, the condition corresponding to (\ref{PJC}) is simply
\begin{equation}\label{pjc_general_f}
\Delta L_t = \inf \bigl\{ x > 0 : \nu_{t-}\bigl( \hspace{0.1pt} \bigl[ 0,\,\alpha\cdot(f(x+L_{t-}) - f(L_{t-})) \bigr]\hspace{0.1pt} \bigr) < x \bigr\}.
\end{equation}

\subsection{PDE viewpoint in the Brownian case}\label{subsec:PDE}
We now turn to a brief review of existing PDE approaches to the fundamental setting of $Z_t=B_t$ and $f(x)=x$, where $B$ is a standard Brownian motion. If we let $V_t$ denote the density of $\nu_t$, i.e.~the law of $X_t$ absorbed at the origin, then we arrive (at least formally) at the PDE
\begin{equation}\label{PDE}
\partial_{t}V_{t}(x)={\textstyle \frac{1}{2}}\partial_{xx}V_{t}(x)+\alpha L_{t}^{\prime}\partial_{x}V_{t}(x),\qquad L_{t}^{\prime}={\textstyle \frac{1}{2}}\partial_{x}V_{t}(0),\qquad V_{t}(0)=0,
\end{equation}
for $x\in(0,\infty)$. Here the contagious feedback emerges as a transport term that pushes mass towards the origin at a rate proportional to the current flux across the boundary.

Setting $v(t,x):=-\alpha V_t(x-\alpha L_t)$, for $\alpha>0$, the equations for $V$ and $L$ 
become 
\begin{equation}\label{stefan_prob}
\partial_tv=\textstyle{\frac{1}{2}}\partial_{xx}v \;\; \text{on} \;\; (\alpha L_t,\infty), \qquad \partial_xv(t,\alpha L_t)=-\textstyle{\frac{1}{2}} \alpha L'_t,\qquad v(t,\alpha L_t)=0.
\end{equation}
This is a Stefan problem modelling the freezing of a supercooled liquid occupying the semi-infinite strip $(\alpha L_t,\infty)$: the liquid is initially supercooled to a temperature $v(0,x)=-\alpha V_0(x)$ that is below the freezing point $v=0$, and the evolving `freezing front' is given by $x=\alpha L_t$. Forgetting for a moment that $V_0$ is a probability density, if $v(0,\cdot)=-c$, then the well-posedness of this system displays a clear dichotomy: for $c<1$ it admits an explicit similarity solution, while no solution can exist for $c\geq1$. This situation motivates the recent analysis of Dembo and Tsai \cite{dembo_tsai_2017}, which investigates the critical case $c=1$ via the scaling behaviour of a discretised particle approximation.

A related line of study is \cite{cell_1,cell_2,cell_3}, which considers the PDE
\begin{equation}\label{keller_segel}
\partial_{t} u(t,x)=\textstyle{\frac{1}{2}}\partial_{xx}u(t,x)+ u(t,0)\partial_{x}u(t,x), \qquad\partial_x u(t,0)=-u(t,0)^2,
\end{equation}
for $x\in(0,\infty)$, as a model for cell polarisation: $u$ is the density of molecular markers on a cell, identified with the positive half-line, and the time evolution of the density is coupled with the concentration of markers on the cell membrane at $x=0$. If we set $\tilde{V}_t(x):=\alpha^{-1}\!\int_0^x\!u(t,y)dy$, then we get back (\ref{PDE}) with initial condition $\tilde{V}_0(x)=\alpha^{-1}\!\int_0^x\!u_0(y)dy$. Calvez et al.~\cite{cell_1,cell_2} show that (\ref{keller_segel}) admits a global weak solution if $\int_0^\infty\!u_0(y)dy\leq1$, while $u$ explodes to infinity in finite time if $\int_0^\infty\!u_0(y)dy>1$ with $u_0$ non-increasing (see also \cite{cell_3} for some refinements).

Note that in both of the above examples, the solvability depends critically on how the initial condition compares to $\alpha^{-1} $. This same relationship will play an important role in the results of the next sections. Returning for now to the supercooled Stefan problem, there is an early literature on its well-posedness for slight variations of (\ref{stefan_prob}) on a finite strip. This goes back to Fasano and Primicerio \cite{fasano_primicerio_1981,fasano_primicerio_1983}, who give conditions on the initial datum under which the system is uniquely solvable in the class of classical solutions for all times or up to some finite explosion time.

If a source term $L'_t\delta_0(x-c)$ is added to (\ref{PDE}), for some $c>0$, where $\delta_0$ is the delta function, then the mass of the system is preserved. Recalling the finite particle systems described earlier, this corresponds to the case where particles are instantly reset to a predetermined value $c$ upon reaching the boundary (instead of being absorbed). By reducing the analysis to that of a Stefan-like problem, Carillo et al.~\cite{carillo_et_al_2013} show that, for $\alpha\leq0$, a unique classical solution exists for all times, while, for $\alpha>0$, classical solutions exists up to a possibly finite explosion time (for initial conditions that are $C^1$ up to the boundary and vanish there). For further background, see also \cite{caceres_et_al_2011, carillo_et_al_2015}.

At this point, it is important to emphasise that the solutions to (\ref{MV}) and (\ref{CMV}) from \cite{dirt_SPA_2015, ledger_sojmark_2018,nadtochiy_shkolnikov_2018} are global in time: they do not cease to exist at some explosion time, even despite the fact that there must be blow-ups in the form of jump discontinuities for large enough $\alpha>0$ \cite[Thm.~1.1]{hambly_ledger_sojmark_2018}.

\subsection{Recent history of the problem in the Brownian case}\label{subsec:recent_history}

Aside from the existence results \cite{ledger_sojmark_2018,nadtochiy_shkolnikov_2018}, the literature on (\ref{MV}) is centred on the case where $Z$ is a Brownian motion, $B$, up to an absolutely continuous drift. For clarity, we thus focus on $Z_t=B_t$ and $f(x)=x$ in this subsection. Motivated by the incomplete results for the PDE viewpoint, when $\alpha>0$, Delarue et al.~\cite{dirt_annalsAP_2015,dirt_SPA_2015} introduced what is essentially a generalised probabilistic notion of solution for the PDE problem \eqref{PDE}, namely (\ref{MV}). Recently, the study of uniqueness and regularity for this problem has been continued through independent efforts of Nadtochiy and Shkolnikov \cite{nadtochiy_shkolnikov_2017} and the authors of this paper together with Hambly \cite{hambly_ledger_sojmark_2018}. Here is a brief overview of the results:
\begin{itemize}
	\item Let $X_0=x_0>0$. By \cite{dirt_annalsAP_2015} there is an $\alpha_0\in(0,1]$ such that, for any $\alpha\in(0,\alpha_0)$, (\ref{MV}) has a $C^1$ solution $L$ on any time-interval, and the solution is unique in this class. The result is formulated for the (neuronal) version where particles are reset upon hitting the boundary.
	\item In the same setting as \cite{dirt_annalsAP_2015}, for any $\alpha>0$, \cite{dirt_SPA_2015} obtains global solutions to (\ref{MV}) as limit points of the (neuronal) particle system described earlier. Moreover, there is propagation of chaos provided there is uniqueness among the limit points.
	\item Let $V_0\in H_0^1(0,\infty)$. For any $\alpha>0$, \cite{nadtochiy_shkolnikov_2017} gives a solution $L$ to (\ref{MV}) up to an explosion time $t_\star>0$ such that $L'\in L^2(0,t)$ for $t<t_\star$ and $\Vert L' \Vert_{L^2(0,t)}$ explodes as $t\uparrow t_\star$. Moreover, there is uniqueness up to $t_\star$ in this class of solutions.
	\item Let $V_0\in L^\infty$ and $V_0(x)\leq Cx^\beta$ for some $C,\beta>0$. For any $\alpha>0$, \cite{hambly_ledger_sojmark_2018} gives a solution $L$ to (\ref{MV}) up to an explosion time $t_\star>0$ such that $L'\in L^2(0,t)$ for $t<t_\star$ and $\Vert L' \Vert_{L^2(0,t)}$ explodes as $t\uparrow t_\star$. Moreover, there is uniqueness up to $t_\star$ among all candidate solutions, and $|L'_t|\leq K_0t^{-\frac{1-\beta}{2}}$ on $[0,t_0]$ for any $t_0<t_\star$.
\end{itemize}

Note that the final result can be combined with that of the second bullet point to give propagation of chaos for (\ref{CMV}, $\rho=0$) up to the explosion time (or globally for small $\alpha>0$). Based on the above results, numerical schemes for (\ref{CMV}, $\rho=0$) up to the explosion time have been proposed and analysed by Kaushansky and Reisinger \cite{vadim_reisinger_2018} and Kaushansky, Lipton and Reisinger~\cite{lipton_vadim_reisinger_2018}.

\subsection{Overview of the paper and main results}

In Section 2 we prove global uniqueness of (\ref{MV}) under a smallness condition on $\alpha$ (Theorem \ref{Results_Thm_UniquenessSimple}) which defines the weak feedback regime. Based on this, we then show that (\ref{CMV}) is globally well-posed in the weak feedback regime (Theorem \ref{Results_Thm_CMV}). In Section 3 we specialise to (\ref{MV}) with $Z_t=B_t$ and $f(x)=x$ for a general $\alpha>0$ without the smallness condition. In this setting, we extend the techniques from Section 2 to give short-time uniqueness under a mild assumption on the shape of the initial density near zero (Theorem \ref{Intro_Prop_ShortTimeUniqueness}). Finally, we use this result to prove local uniqueness of the problem after a blow-up (Theorem \ref{thm_global_uniqafterblowup}).


\section{Global uniqueness for weak feedback}

The main objective of this section is to prove uniqueness of solutions to (\ref{MV}) in the weak feedback regime, that is, when the feedback parameter  $\alpha$ satisfies a suitable smallness condition (Theorem \ref{Results_Thm_UniquenessSimple}). Our method of proof is based on a surprisingly simple comparison argument, which is of `zeroth order' in the sense that it makes no use of differential or analytic properties of the solution, $L$, as a function of time or of the initial density, $V_0$, as a function of space. This is very natural probabilistically, as the McKean--Vlasov formulation (\ref{MV}) does not involve any derivatives---in contrast to the PDE point of view \eqref{PDE}. The value of our zeroth order approach is most evident when there is a rough drift in the driving noise, $Z$, as the analytical tools developed in the earlier literature \cite{carillo_et_al_2013, dirt_annalsAP_2015, hambly_ledger_sojmark_2018, nadtochiy_shkolnikov_2017} cannot be applied to such a setting.

\subsection{Uniqueness in the weak feedback regime}\label{subsec:unique_weak}

Recall that we are interested in the McKean--Vlasov system \eqref{MV} under the general hypothesis that
\begin{equation}\tag{H}\label{hyp_H}
  Z \text{ is a continuous stochastic process}, \quad Z\perp X_0, \quad \text{and} \quad f\in C([0,1],\mathbb{R}).
\end{equation}
Our first task is to carry out the comparison argument alluded to above. We present it as a lemma separately from Theorem \ref{Results_Thm_UniquenessSimple}, since we will need it again in Section 3. A visualisation of its proof is provided in Figure \ref{fig:visual}, which serves to highlight the geometric flavour of both this lemma and the resulting Theorem \ref{Results_Thm_UniquenessSimple}.

\begin{lem}[Comparison argument]
\label{Lem_Unique_Comparison}
Suppose $L$ and $\bar{L}$ are two solutions to \eqref{MV} under hypothesis \eqref{hyp_H} with initial condition $\nu_0$. Then
\begin{multline*}
| L_t - \bar{L}_t |
	\leq \mathbb{E}\Bigl[\nu_{0} \Bigl( \sup_{s \leq t} \{ \alpha f(L_s) - Z_s \}, \, \sup_{s \leq t} \{  \alpha f(\bar{L}_s) - Z_s \} \Bigr) \Bigr]\\
	\vee 
	\mathbb{E}\Bigl[ \nu_{0} \Bigl( \sup_{s \leq t} \{ \alpha f(\bar{L}_s) - Z_s \}, \, \sup_{s \leq t} \{ \alpha f( L_s) - Z_s \}\Bigr) \Bigr], 
\end{multline*}
where `$\,\lor$' denotes the maximum of the two values on the right-hand side.
\end{lem}

\begin{proof} Let $L$ and $\bar{L}$ be any two solutions to (\ref{MV}) under hypothesis \eqref{hyp_H} coupled via the same stochastic process $Z$ and the same random starting point $X_0$ distributed according to $\nu_0$. That is, the two solutions $L$ and $\bar{L}$ are given by
\[
\begin{cases}
 X_t = X_0 + Z_t - \alpha f(L_t) \\
\tau = \inf\{ t \geq 0 : X_t \leq 0 \}  \\
L_t = \mathbb{P}(\tau \leq t),
\end{cases}
\quad  \text{and} \quad
\begin{cases}
 \bar{X}_t = X_0 + Z_t - \alpha f(\bar{L}_t) \\
\bar{\tau} = \inf\{ t \geq 0 : \bar{X}_t \leq 0 \}  \\
\bar{L}_t = \mathbb{P}(\bar{\tau} \leq t).
\end{cases}
\]
Then we have
\begin{align*}
L_t - \bar{L}_t 
	&\leq \mathbb{P}\bigl(\inf_{s\leq t} X_s \leq 0, \, \inf_{s\leq t} \bar{X}_s > 0\bigr) \\
	&= \mathbb{P}\Bigl(\inf_{s\leq t} \{ X_0 + Z_s -\alpha f(L_s)\} \leq 0,\, \inf_{s\leq t} \{ X_0 + Z_s -\alpha f( \bar{L}_s )\} > 0\Bigr) \\
	&= \mathbb{P}\Bigl(\sup_{s\leq t} \{ \alpha f(\bar{L}_s) - Z_s\} < X_0 \leq \sup_{s\leq t} \{ \alpha f(L_s) - Z_s\} \Bigr).
\end{align*}
In turn, conditioning on $X_0$ gives
\begin{align*}
L_t - \bar{L}_t 
	&\leq \mathbb{E} \Bigl[ \int^\infty_0 \mathbf{1}\bigl\{ \sup_{s\leq t} \{ \alpha f( \bar{L}_s ) - Z_s\} < x_0 \leq  \sup_{s\leq t} \{ \alpha f(L_s) - Z_s\} \bigr\}  \nu_0(dx_0) \Bigr] \\
	&= \mathbb{E} 
	\Bigl[ \nu_0 \Bigl(\sup_{s\leq t} \{ \alpha f( \bar{L}_s ) - Z_s\},\, \sup_{s\leq t} \{ \alpha f(L_s) - Z_s\}  \Bigr)  \Bigr].
\end{align*}
Finally, we can notice that the same inequality with $L$ and $\bar{L}$ interchanged also holds, by symmetry, and therefore we have the result. 
\end{proof}

\begin{figure}
\begin{center}
\includegraphics[width=0.32\textwidth]{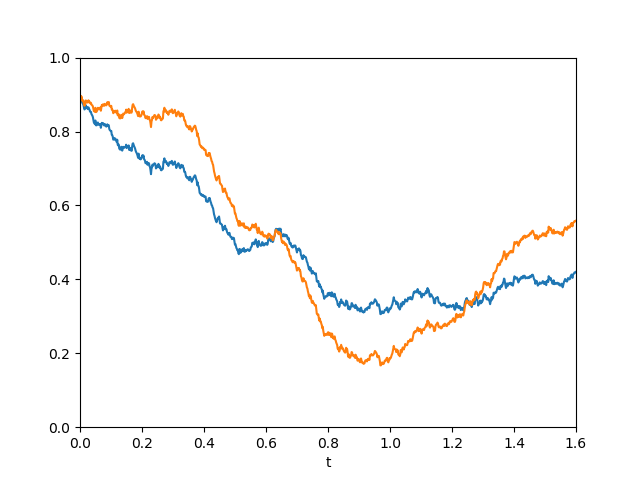} \hspace{-0.5cm}
\includegraphics[width=0.32\textwidth]{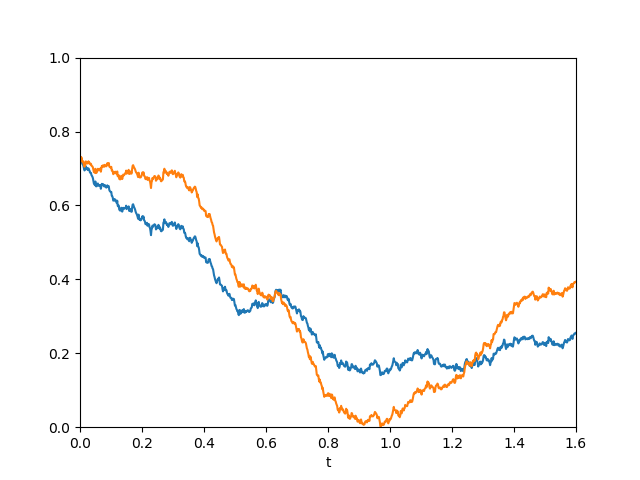} \hspace{-0.5cm}
\includegraphics[width=0.32\textwidth]{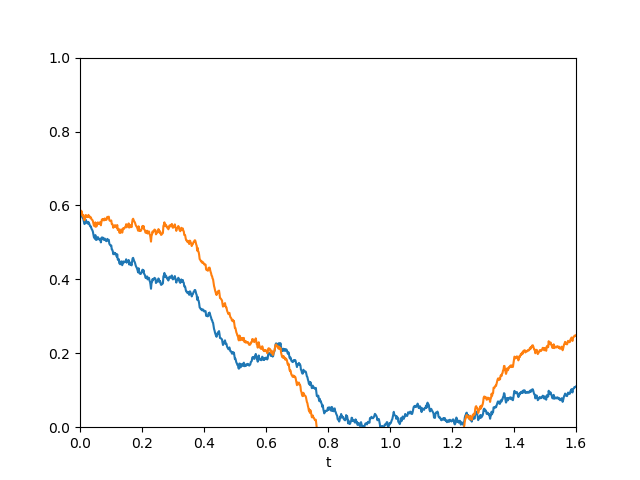} 
\caption{\label{fig:visual}The red curve is $X$ and the blue curve is $\bar{X}$, from Lemma \ref{Lem_Unique_Comparison}, coupled through the initial value $x_0$ and a given realisation of the driver $Z$. Writing $L_t - \bar{L}_t \leq \mathbb{P}( \tau \leq t < \bar{\tau})$ with $t=1.6$, we see that, as $x_0$ decreases, we only get contributions to this probability from when the red curve first hits the $x$-axis (at $x_0=0.72$ in the second plot) until the blue curve first hits it (at $x_0=0.59$ in the third plot). Hence the probability equals $\nu_0(0.59,0.72)$, where $0.59= \sup_{s\leq t}\{\alpha f(\bar{L}_s) - Z_s\}$ and $0.72= \sup_{s\leq t}\{\alpha f(L_s) - Z_s\}$. In terms of Theorem \ref{Results_Thm_UniquenessSimple}, one can see directly that $0.72-0.59\leq \alpha \Vert f(L) - f(\bar{L})\Vert_t$, since the latter bounds the distance between the two curves on $[0,t]$.}
\end{center}
\end{figure} 

With a view towards the statement and proof of Theorem \ref{Results_Thm_UniquenessSimple} below, it is convenient to introduce the notation
\[
\Vert f \Vert_{t} := \sup_{s \leq t} |f(s)| \quad \text{and} \quad 	\Vert f \Vert_{ \mathrm{Lip}(x) } := \sup_{ y \neq z \in [0,x] } \frac{|f(y) - f(z)|}{|y-z|}. 
\]
Theorem \ref{Results_Thm_UniquenessSimple} below is an almost immediate consequence of Lemma \ref{Lem_Unique_Comparison}, once we place suitable restrictions on the relation between the inputs. Specifically, we introduce a `smallness' condition (\ref{alpha_sc}) on the feedback parameter $\alpha$ given $f$ and $\nu_0$, which defines the \emph{weak feedback regime} of \eqref{MV} for any choice of $Z$.

\begin{thm}[Uniqueness in the weak feedback regime] 
\label{Results_Thm_UniquenessSimple}
Let $L$ and $\bar{L}$ be any two solutions to \eqref{MV} under hypothesis \eqref{hyp_H}, and suppose the initial condition $\nu_0$ has a density $V_0 : (0,\infty) \to (0,\infty)$.  If the feedback parameter $\alpha \in \mathbb{R} $ is such that
\begin{equation}\label{alpha_sc}
|\alpha| \cdot \Vert V_0 \Vert_\infty \cdot \Vert f \Vert_{ \mathrm{Lip}(L_{T} \vee \bar{L}_{T}) } < 1,
\end{equation}
for some $T>0$, then $L=\bar{L}$ on $[0,T]$. In particular, if $f$ has a global Lipschitz constant and $V_0$ is bounded, then \eqref{alpha_sc} gives a range of $\alpha$ for which there is global uniqueness of \eqref{MV} independently of the choice of $Z$.
\end{thm}
\begin{proof} We only need to consider $\alpha\geq0$, as we can otherwise absorb the minus sign into $f$. Let $L$ and $\bar{L}$ denote any two solutions to (\ref{MV}), and observe that
\[
f(\bar{L}_s) = f(L_s) + f(\bar{L}_s) - f(L_s) \geq f(L_s) - \Vert f(L) - f(\bar{L}) \Vert_r \qquad \text{for } s \leq r.
\]
Applying this inequality to the bound from Lemma \ref{Lem_Unique_Comparison} gives
\begin{multline*}
| L_r - \bar{L}_r |
	\leq \mathbb{E}\bigl[\nu_{0}\bigl( - \alpha \Vert f(L) - f(\bar{L}) \Vert_r + \sup_{s \leq r} \{ \alpha f(L_s) - Z_s \}, \,\sup_{s \leq r} \{ \alpha f(L_s) - Z_s \}\bigr) \bigr]\\
	\vee 
	\mathbb{E}\bigl[ \nu_{0}\bigl( - |\alpha| \Vert f(L) - f(\bar{L}) \Vert_r + \sup_{s \leq r} \{ \alpha  f(\bar{L}_s) - Z_s \},\, \sup_{s \leq r} \{ \alpha f(\bar{L}_s) - Z_s \}\bigr) \bigr].
\end{multline*}
Writing $Z^*_r := \sup_{s \leq r} \{ \alpha f(L_s) - Z_s \}$ and $\bar{Z}^*$ for the same running supremum corresponding to $\bar{L}$, the above becomes
\[
| L_r - \bar{L}_r | \leq \mathbb{E} \Big[ \int^{ Z^*_r }_{ Z^*_r - \alpha \Vert f(L) - f(\bar{L}) \Vert_r } \!V_0(x) dx \Big]
	\vee  \mathbb{E}\Big[  \int^{ \bar{Z}^*_r }_{ \bar{Z}^*_r - \alpha \Vert f(L) - f(\bar{L}) \Vert_r } \!V_0(x) dx \Big],
\] 
and so it is immediate that, for all $r\geq0$,
\[
| L_r - \bar{L}_r | 
	\leq \alpha \Vert V_0 \Vert_\infty \Vert f(L) - f(\bar{L}) \Vert_r. 
\]
Using the local Lipschitz property of $f$, as implied by \eqref{alpha_sc}, and taking a supremum over $r \in [0,T]$, we then get
\[
\Vert L - \bar{L} \Vert_{T}
	\leq \alpha \Vert V_0 \Vert_\infty \Vert f \Vert_{ \mathrm{Lip}( L_{T} \vee \bar{L}_{T} ) }  \Vert L - \bar{L} \Vert_{T}.
\]
Therefore, the smallness condition \eqref{alpha_sc} forces $\Vert L - \bar{L} \Vert_T = 0$, as required. This completes the proof.
\end{proof}
Fix a choice of the inputs $f$, $V_0$, and $\alpha$ such that the smallness condition (\ref{alpha_sc}) holds on some interval $[0,T]$, and suppose the stochastic driver $Z_t$ has a density $p_t$. Then $\nu_t$ has a density $V_t$ with $\Vert V_t \Vert_{\infty}\leq \Vert V_0 \Vert_{\infty} $, as can be seen from the simple estimate
	\begin{equation}\label{density_bound}
	\nu_{t}(S)\leq \int_0^\infty \int_S p_t(x_0+z-\alpha f(L_t)) V_0(x_0)dzdx_0 \leq \Vert V_0 \Vert_{\infty} \cdot |S|,
	\end{equation}
for all $S\in\mathcal{B}(\mathbb{R})$. Therefore, the smallness condition (\ref{alpha_sc}) enforces $\Delta L_t=0$ for all $t\in[0,T]$, by virtue of the physical jump condition \eqref{PJC}, so the unique solution to \eqref{MV} in the weak feedback regime on $[0,T]$ is continuous on all of $[0,T]$.

\subsection{Applications of the main uniqueness result}\label{subsec:applications_unqiue}

We now present two interesting consequences of Theorem \ref{Results_Thm_UniquenessSimple}.

\begin{thm}[Well-posedness of \eqref{CMV} for weak feedback]
\label{Results_Thm_CMV} Consider the conditional McKean--Vlasov system \eqref{CMV} with $\rho\in[0,1)$ and $\alpha>0$ under the constraint $\Vert V_0 \Vert_\infty < \alpha^{-1}$. Then there is a unique solution to \eqref{CMV}, and this solution arises as the unique mean-field limit of the finite particle system from \emph{\cite[(3.1)]{ledger_sojmark_2018}}.
\end{thm}
\begin{proof} By \cite[Theorem 3.2]{ledger_sojmark_2018} there exists a `relaxed' solution $(\bar{X},\bar{L},\bar{\mathbf{P}})$ to \eqref{CMV} given by
\[
\begin{cases}
\bar{X}_t = X_0 + \sqrt{1-\rho^2} B_t + \rho B^0_t - \alpha \bar{L}_t, \\
\bar{L}_t = \mathbb{P}(\bar{\tau} \leq t \, | \, B^0,\bar{\mathbf{P}}), \;\; \bar{\tau} = \inf\{ t \geq 0 : \bar{X}_t \leq 0 \},  \\
\bar{\mathbf{P}} =\text{Law}(X \,|\,B^0,\bar{\mathbf{P}}), \;\; (B^0,\bar{\mathbf{P}}) \perp B,
\end{cases}
\]	
for a 2d Brownian motion $(B,B^0)$ and initial condition $X_0\perp (B,B^0,\bar{\mathbf{P}})$, where we note that $\bar{\mathbf{P}}$ is a random probability measure on the space of c\`adl\`ag paths. Let $(\widetilde{X},\widetilde{L}, \widetilde{\mathbf{P}})$ be another relaxed solution coupled to $(\bar{X},\bar{L}, \bar{\mathbf{P}})$ through the same Brownian drivers $(B,B^0)$  and the same random start point $X_0$.
When comparing $\widetilde{L}$ and $\bar{L}$, the conditioning fixes a pathwise realisation of $B^0$, $\bar{L}$, and $\widetilde{L}$, so we can apply Theorem \ref{Results_Thm_UniquenessSimple} with $Z_t=\sqrt{1-\rho^2} B_t + \rho \beta_t$ for a given realisation $B^0=\beta$.
This proves the pathwise equality $\widetilde{L}=\bar{L}$, and hence also $\widetilde{X}=\bar{X}$ and $\widetilde{\mathbf{P}}=\bar{\mathbf{P}}$. From here, a Yamada--Watanabe argument (see \cite{kurtz_ecp_2014}) gives existence of a relaxed solution $(X,L,\mathbf{P})$ that is $(X_0,B,B^0)$-measurable. But the definition of a relaxed solution entails that $(X_0,B,B^0)$ is conditionally independent of $\mathbf{P}$ given $B^0$, so we get $L_t=\mathbb{P}(t\geq\tau\,|\,B^0,\mathbf{P})=\mathbb{P}(t\geq\tau\,|\,B^0)$, and hence $(X,L)$ is a bonafide solution to \eqref{CMV}.
 Finally, \cite[Thm.~3.2]{ledger_sojmark_2018} shows that the limit points of the particle system \cite[(3.1)]{ledger_sojmark_2018} are supported on relaxed solutions to \eqref{CMV}, but these must now agree with the unique bonafide solution $(X,L)$. Hence there is full weak convergence of the particle system to this unique limit.
\end{proof}

 To compare with the PDE point of view in Section \ref{subsec:PJC}, let $V_t$ be the random density function of $\nu_t=\mathbb{P}(X_t\in\!\cdot\,,t<\tau \, | \, B^0)$. Then the system (\ref{CMV}) gives rise to a stochastic PDE
\[
dV_t(x) = \tfrac{1}{2} \partial_{xx}V_t(x) dt + \alpha \partial_{x}V_t(x) dL_{t} - \rho \partial_{x}V_t(x) dB_{t}^{0}, \quad V_t(0)=0,
\]
for $(x,t)\in(0,\infty)\times[0,\infty)$, where $L_t=1-\int_0^\infty V_t(x)dx$.  One can formally integrate by parts to find that $L^\prime_t=\frac{1}{2}\partial_xV_t(0)$, but unless $\rho=0$ we can no longer expect $L$ to be differentiable and the derivative of $V_t(\cdot)$ at zero fails to be defined.

As a final application of Theorem \ref{Results_Thm_UniquenessSimple}, we present a local uniqueness result for \eqref{MV} with $f(x)=-\log(1-x)$, as studied in \cite{nadtochiy_shkolnikov_2017, nadtochiy_shkolnikov_2018}, for a general continuous driver $Z$.

\begin{cor}
	\label{Results_Thm_Z}
	Consider \eqref{MV} with $f(x)=-\log(1-x)$ and $\alpha>0$, for any continuous driver $Z$, and suppose $\Vert V_0 \Vert_\infty<\alpha^{-1}$. Then there is uniqueness of solutions to \eqref{MV} on $[0,t]$ for all $t\geq0$ such that $L_t<1-\alpha\Vert V_0 \Vert_\infty$.
\end{cor}
\begin{proof}
	Note that $\Vert f \Vert_{\mathrm{Lip}(x)}=(1-x)^{-1}$, so the result follows from Theorem \ref{Results_Thm_UniquenessSimple}.
	\end{proof}

We should note that Theorem \ref{Results_Thm_UniquenessSimple}, and the two above applications, represent the most straightforward utilization of the ideas from Lemma \ref{Lem_Unique_Comparison}. Indeed, it is possible to obtain stronger results with harder work and a more specific setting. The next section gives one such example for the Brownian case.


\section{Local uniqueness after a blow-up}


For the remaining part of the paper, we return to the setting $Z_t=B_t$ and $f(x)=x$, where $B$ is a standard Brownian motion, and we focus on the case $\alpha>0$, for which blow-ups may occur. As highlighted in Section \ref{subsec:recent_history}, there is full uniqueness of (\ref{MV}) up to an explosion time $t_\star>0$, and we know that there exist solutions for all time (for a natural class of initial conditions); however, the results in this section are the first to address uniqueness of the restarted system after a blow-up.

The earlier approaches to uniqueness (see Section \ref{subsec:recent_history}) break down at the first explosion time $t_\star$, since the system may be restarting from a density $V_{t_\star}$ which is not sufficiently well-behaved at the origin. In principle, all we know is that (\ref{PJC}) imposes
\begin{equation}
\label{eq:global_I}
\inf\{x>0:\textstyle{\int}_0^{\alpha x} V_{t_\star}(y)dy<x\}=0.
\end{equation}
The problem here is that (\ref{eq:global_I}) implies little about the  regularity of $V_{t_\star}$ near zero. Without further information, we cannot rule out pathological cases like those in Figure \ref{fig_global_badguys}, so it seems difficult to gain sufficient control to prove that uniqueness can always be propagated. In practice, however, we do not expect these edge cases to arise, and indeed we are able to prove here that we have at least polynomial control on the density after a blow-up  (defined as a jump time of $L$). This observation is derived from an analyticity result for the left limit density $V_{t_\star-}$ (Proposition \ref{prop:analyticity}) and it allows us to prove short-time uniqueness after a blow-up (Theorem \ref{thm_global_uniqafterblowup}), by applying Theorem \ref{Intro_Prop_ShortTimeUniqueness} which we present next.

\subsection{Short-time uniqueness for general feedback}

As the next result shows, polynomial control on the initial density near the origin is sufficient to have short-time uniqueness. The idea of the proof is to use the methods from the previous section, but to ensure that \emph{insufficient} time has passed for a large amount of mass to reach the boundary, thus counteracting the effect of the density possibly being above $\alpha^{-1}$ away from the origin as well as allowing $V_{0}(0+)=\alpha^{-1}$.

\begin{thm}[Short-time uniqueness]
\label{Intro_Prop_ShortTimeUniqueness} Suppose the intial condition $\nu_0$ has a density, $V_0$, for which there exists $c>0$, $x_0>0$ and $n \in \mathbb{N}$ such that
\begin{equation}
\label{eq:global_IC}
V_0(x) \leq \alpha^{-1} - cx^n ,
	\qquad \textrm{ for all } x < x_0. 
\end{equation}
Let $L$ and $\bar{L}$ be two solution to \eqref{MV} with $Z_t=B_t$ and $f(x)=x$, where $B$ is a Brownian motion. Then there exists $t_0 > 0$ such that $L_t = \bar{L}_t$ for all $t \in [0,t_0]$.
\end{thm}
\begin{proof}
	For shorthand let $Z^*_t := \sup_{s \leq t} \{ \alpha L_s - B_s \}$ and note that 
	\begin{equation}\label{eq:short_unique_polyno}
	\mathbb{E}\bigl[ \nu_0([-\alpha \Vert L-\bar{L} \Vert_t + Z^*_t, Z^*_t)) \bigr]
	\leq E(t) \cdot \alpha \Vert L - \bar{L} \Vert_t,
	\end{equation}
	where
	\[
	E(t) := \mathbb{E}\Bigl[ \, \sup \bigl\{V_0(x)  :x \in [-\alpha \Vert L - \bar{L} \Vert_t + Z^*_t, Z^*_t) \bigr\} \Bigr].
	\]
	We decompose this latter expectation into the three regions:
	\[
	Z^*_t \in [0, \alpha \Vert L - \bar{L} \Vert_t + t),
	\qquad Z^*_t \in [\alpha \Vert L - \bar{L} \Vert_t + t, x_0),
	\qquad Z^*_t \in [x_0,\infty).
	\]
	Since $L$ and $\bar{L}$ are c\`adl\`ag, we can take $t_1 > 0$ sufficiently small so that for all $t \leq t_1$ we have $\alpha \Vert L - \bar{L} \Vert_t + t \leq x_0/2$ and $\alpha L_t, \alpha \bar{L}_t \leq x_0/4$. Therefore
	\begin{align*}
	E(t) 
	&\leq \alpha^{-1} \mathbb{P}(Z^*_t \in [0,\alpha \Vert L - \bar{L} \Vert_t + t]) 
	+ (\alpha^{-1} - c t^n ) \mathbb{P}(Z^*_t \in [\alpha \Vert L - \bar{L} \Vert_t + t, x_0]) \\
	&\qquad + \Vert V_0 \Vert_\infty \mathbb{P}(Z^*_t \in [x_0, \infty)) \\
	&\leq  \alpha^{-1} - c t^n \mathbb{P}(Z_t \in [x_0/2, x_0])  + \Vert V_0 \Vert_\infty \mathbb{P}(Z^*_t \in [x_0, \infty)) \\
	&\leq  \alpha^{-1} - c t^n \mathbb{P}( \sup_{s \leq t} B_s \in [\tfrac{1}{2} x_0, \tfrac{3}{4} x_0])  + \Vert V_0 \Vert_\infty \mathbb{P}( \sup_{s \leq t} B_s \in [\tfrac{3}{4} x_0, \infty)) \\
	&= \alpha^{-1} - ct^n ( \Phi(-x_0/2t^{1/2}) - \Phi(-3x_0/4t^{1/2}) ) + \Vert V_0 \Vert_\infty \Phi(-3x_0/4t^{1/2}),
	\end{align*}
	where $\Phi$ is the standard normal cdf. Using the asymptotic bounds
	\[
	\Phi(-cx) \asymp x^{-1} e^{-c^2 x^2 / 2},
	\qquad \textrm{as } x \to \infty, 
	\] 
	we can find $t_2 \in (0, t_1]$ sufficiently small so that $E(t) < \alpha^{-1}$ for all $t \leq t_2$. Therefore, recalling \eqref{eq:short_unique_polyno}, it follows by symmetry and Lemma \ref{Lem_Unique_Comparison} that
	\begin{equation}
	\label{eq:Unique_MainProof_I}
	|L_t - \bar{L}_t| \leq (1-\varepsilon) \Vert L - \bar{L} \Vert_t,
	\qquad \textrm{for every } t \in (0,t_2],
	\end{equation}
	for some $\varepsilon\in(0,1]$. Now let $t_0 > 0$ be the smaller of the first jump times of $L$ and $\bar{L}$ and $t_2$. By continuity there exists $s_0 \in [0,t_0]$ for which the supremum on $[0,t_0]$ is attained. That is,
	\[
	 \Vert L - \bar{L} \Vert_{t_0} = \Vert L - \bar{L} \Vert_{s_0} = |L_{s_0} - \bar{L}_{s_0}|.
	\] 
	If $s_0 = 0$, then $\Vert L - \bar{L} \Vert_{t_0} = |L_0-\bar{L}_0|= 0$ and we are done. Otherwise, $s_0 \in (0,t_0]$, but this contradicts (\ref{eq:Unique_MainProof_I}) at $t = s_0$ unless $ \Vert L - \bar{L} \Vert_{s_0}  =0$, so the proof is complete.
\end{proof}

\begin{figure}
	\begin{center}
		\includegraphics[width=0.5 \textwidth]{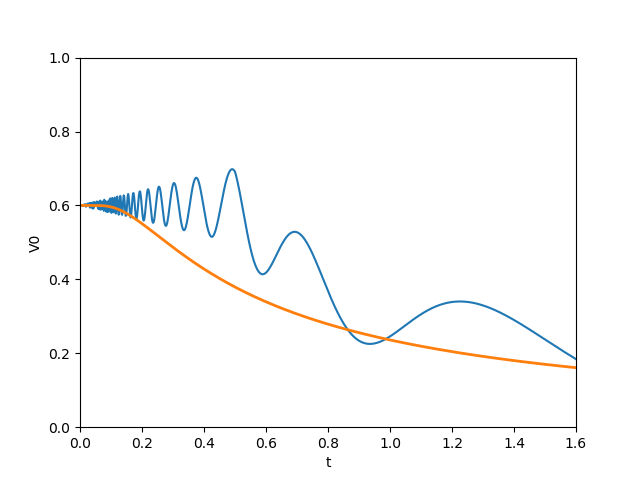} 
		\caption{\label{fig_global_badguys} Two pathological initial densities that satisfy (\ref{eq:global_I}) with $\alpha = 5/3$. The blue curve takes infinitely many values above the critical value of $3/5$ near zero. The red curve remains strictly below $3/5$, but all derivatives vanish at the origin, so control of the form in (\ref{eq:global_IC}) is not possible.  }
	\end{center}
\end{figure} 

It is worth emphasising that the main difficulty in the above proof is concentrated at the initial time $t=0$, where we are faced with $V_0(0+)=\alpha^{-1}$. Indeed, for any small enough $t>0$, the diffusivity (and the c\`adl\`agness of $L$) forces the density $V_t$ strictly below $\alpha^{-1}$ in a small neighbourhood of $x=0$ (see \cite[Prop.~6.4.3]{sojmark_2019}), meaning that we are essentially back in the weak feedback regime of Theorem \ref{Results_Thm_UniquenessSimple}, albeit in a local sense near the origin. By similar observations, one can use Theorem \ref{Intro_Prop_ShortTimeUniqueness} to get global uniqueness for densities that look like (\ref{eq:global_IC}) near the origin but lie above $\alpha^{-1}$ sufficiently far out so that the diffusivity forces it below $\alpha^{-1}$ before it can cause a blow-up.

\subsection{Short-time uniqueness after a blow-up}

In what follows we will show how to propagate the conclusion of Theorem \ref{Intro_Prop_ShortTimeUniqueness} to blow-up times. The main point is that, although we cannot control the solution density near zero at arbitrary times, we can prove that the density is analytic in the interior of the half-line. At a blow-up time this is then sufficient to give control near zero of the new density from which the system restarts, since the new point at the origin was in the interior of the density before the jump discontinuity (see the proof of Theorem \ref{thm_global_uniqafterblowup} below). Our proof of analyticity relies on kernel smoothing and energy estimate techniques as used in \cite{hambly_ledger_2017, hambly_sojmark_2018}, but we only state the result here and postpone the proof to the next section.

\begin{prop}[Interior analyticity]
\label{prop:analyticity}
Suppose the initial condition $\nu_0$ has a density $V_0 \in L^2(0,\infty)$ and assume there is a solution to \eqref{MV} with $Z_t=B_t$ and $f(x)=x$, for which we define $\nu_t:=\mathbb{P}(X_t\in \cdot \,, t<\tau)$. Then, for all $t>0$, $\nu_t$ has a bounded density $V_t$ and $y \mapsto V_{t-}(y)$ is analytic at every point $x\in(0,\infty)$. 
\end{prop}

Exploiting the interior analyticity, we are now in a position to prove  short-time uniqueness for the restarted system after a blow-up time.

\begin{thm}[Short-time uniqueness after blow-up]
\label{thm_global_uniqafterblowup}
Let the initial condition $\nu_0$ have a density $V_0 \in L^2(0,\infty)$ and suppose we have a solution $L$ to (\ref{MV}) with $Z_t=B_t$ and $f(x)=x$ up to its first blow-up time $t_\star>0$, where $B$ is a Brownian motion. Then the system can be restarted at time $t_\star$ and the restarted solution is unique (in the class of c\`adl\`ag solutions) on a small time interval $[t_\star,t_{\star\star}]$, for some $t_{\star\star}>t_\star$.
\end{thm}

\begin{proof}
	Note that, at the first blow-up time $t_\star$, we have
	\[
	\nu_{t_\star}(S)=\mathbb{P}(X_{t_\star-}-\alpha \Delta L_{t_\star} \in S) = \nu_{t_\star-}(S+\alpha \Delta L_{t_\star})= \int_{S} V_{t_\star-}(x+\alpha \Delta L_{t_\star})dx
	\]
	for all $S\in \mathcal{B}(\mathbb{R})$, where $\Delta L_{t_\star}$ is uniquely specified by (\ref{PJC}). Consequently, after the blow-up, the system restarts from the new density $V_{t_\star}$ of $\nu_{t_\star}$ given by
	\[
	V_{t_\star}(x) = V_{t_\star -}(x + \alpha \Delta L_{t_\star}) \qquad \text{for all } x \geq 0.
	\]
	In turn, although $V_{t_\star -}$ is not known to be analytic at $x=0$, it follows from $\Delta L_{t_\star} > 0$ and $V_{t_\star-}$ being analytic in the interior that we indeed have analyticity of the new density $V_{t_\star}$ at $x=0$. Therefore we have a series expansion
	\[
	V_{t_\star}(x) = V_{t_\star}(0) + \sum_{n \geq 1} c_n x^n,
	\qquad \textrm{for every } x \in [0, x_0],
	\]
for some $x_0 > 0$. If $V_{t_\star}(0) < \alpha^{-1}$ we have the required condition on $V_{t_\star}$ by taking $x_0$ sufficiently small.

If $V_{t_\star}(0) > \alpha^{-1}$, then, since the physical jump condition (\ref{PJC}) on $\Delta L$ ensures that $V_{t_\star}$ satisfies (\ref{eq:global_I}), by taking $x$ sufficiently small we have a contradiction. Therefore suppose $V_{t_\star}(0) = \alpha^{-1}$, then by the last case we cannot have $c_n = 0$ for all $n \geq 1$, so let $n_0 := \min\{ n : c_n \neq 0  \}$. For $x_1 > 0$ sufficiently small we have 
	\[
	V_{t_\star}(x) \geq \alpha^{-1} + (c_{n_0} + \varepsilon) x^{n_0},
	\qquad \textrm{for every } x \in [0,x_1],
	\]
	where $|\varepsilon| \leq \tfrac{1}{2} |c_{n_0}|$. Again if $c_{n_0} > 0$ then we contradict (\ref{eq:global_I}), hence $c_{n_0} < 0$ and so we have that $V_{t_\star}$ satisfies the condition (\ref{eq:global_IC}) in Theorem \ref{Intro_Prop_ShortTimeUniqueness}. Consequently, Theorem \ref{Intro_Prop_ShortTimeUniqueness} can be applied up to a small time after the first blow-up time $t_\star$. 
\end{proof}

Although we get small-time uniqueness after a blow-up, it is important to note that global uniqueness is out of reach of the techniques presented in this paper. If the system reaches the pathological states of Figure \ref{fig_global_badguys} at a continuity time, then we cannot propagate our uniqueness argument past this time.
	
In the time between the original submission and the present revision of this paper, global uniqueness of \eqref{MV} has been resolved in precisely the case of $f(x)=x$ and $Z_t=B_t$, by Delarue, Nadtochiy, and Shkolnikov \cite{francois_misha_sergey_2018}. One of their key technical results is to rule out the appearance of oscillating densities that change monotonicity infinitely often near the origin (as depicted in Figure \ref{fig_global_badguys}), provided the system is started from a bounded initial density that does not have this property.

\subsection{Analyticity of the density in the interior}

In this final subsection we present a proof of Proposition \ref{prop:analyticity}. As already mentioned, it goes via kernel smoothing, so for any $\delta > 0$ and any measure $\mu$ on $(0,\infty)$, we define the convolutions
\[
T_\delta \mu(x) := \int_0^\infty G_\delta(x_0, x) \mu(dx_0)
\quad \text{and} \quad T^r_\delta \mu(x) :=  \int_0^\infty G^r_\delta(x_0, x)  \mu(dx_0),
\]
for $x \geq 0$, where the kernels $G_\delta$ and $G^r_\delta$ are, respectively, the absorbing and reflecting Gaussian densities on the positive half-line, given by
\[
G_\delta(x_0,x) := \frac{1}{\sqrt{2 \pi \delta}} \Big\{ e^{-\frac{(x_0-x)^2}{2\delta}} - e^{-\frac{(x_0+x)^2}{2\delta}} \Big\}, \;\; G^r_\delta(x_0,x) := \frac{1}{\sqrt{2 \pi \delta}} \Big\{ e^{-\frac{(x_0-x)^2}{2\delta}} + e^{-\frac{(x_0+x)^2}{2\delta}} \Big\}.
\]
Substituting $\delta \mapsto 2\delta$, these are of course the Dirichlet and Neumann heat kernels on the positive half-line. As a warning to the careful reader, we will occasionally abuse notation and simply write $T_\delta\phi$ or $T^r_\delta\phi$ for the operators applied to the measure whose Radon--Nikodym derivative is the function $\phi$.

Given the flow of measures $\nu_t$ associated to \eqref{MV}, the crux of this section is a local $L^2$ energy estimate for the derivatives of $ T_\delta \nu_t(\cdot)$ derived in Proposition \ref{Analytic_Lem_Inductive_I}. Using this and the two lemmas below, Corollary \ref{Analytic_Cor_Smoothness} shows that $\nu_t$ in fact has a smooth density $V_t$. Finally, we finish the proof of Proposition \ref{prop:analyticity} via the specific form of the energy estimates and Sobolev embedding, which give suitable local pointwise bounds on the derivatives guaranteeing analyticity at all interior points of $(0,\infty)$.

\begin{lem}[Existence of weak derivatives]\label{Kernel_Lem_L2density}
	\label{Kernel_Lem_MainLemma}
	Suppose $\liminf_{\delta\to 0} \Vert \partial_x^n T_\delta \mu \Vert_2 < \infty$. Then $\mu$ has an $n^\textrm{th}$ weak derivative $\partial_x^n \mu \in L^2(0,\infty)$ and $\Vert \partial_x^n T_\delta \mu \Vert_2 \to \Vert \partial_x^n \mu \Vert_2$ as $\delta \to 0$.
\end{lem}
\begin{proof}
	Since $\liminf_{\delta\to 0} \Vert \partial_x^n T_\delta \mu \Vert_2 < \infty$, a standard weak compactness argument gives that $(\partial_x^n T_{\delta} \mu)_{\delta>0}$ is convergent with strong limit $h=\partial_x^n \mu$ in $L^2(0,\infty)$, and $\Vert h \Vert_2 \leq \liminf_{\delta \to 0} \Vert \partial_x^n T_{\delta} \mu \Vert_2$. Moreover, we can check that
	\[
	\partial_x^n T_\delta \mu (x) 
	= \langle \mu ,  \partial_x^n G_\delta(\cdot , x) \rangle
	= (-1)^n \langle \mu, \partial^n_{x_0} G^r_\delta(\cdot , x) \rangle
	= (T^r_\delta (\partial^n \mu))(x)
	= T^r_\delta h.
	\]
	But $T^r_\delta$ is an $L^2$-contraction, so we deduce that	$\limsup_{\delta \to 0} \Vert \partial_x^n T_\delta \mu \Vert_2	\leq \Vert h \Vert_2$, and hence $\Vert \partial_x^n T_\delta\mu\Vert$ converges to $\Vert h \Vert_2$.
\end{proof}

Regardless of the initial condition $\nu_0$, \cite[Prop.~2.1]{hambly_ledger_sojmark_2018} shows that $\nu_t$ has a bounded density $V_t : (0,\infty) \to (0,\infty)$ for all positive times $t>0$. Moreover, if $\nu_0$ has a density in $L^2$, then we have control on the $L^2$ norms of each $V_t$ which will serve as the base case for the induction argument in Corollary \ref{Analytic_Cor_Smoothness} below.
\begin{lem}\label{lem:L2_initial_bound} If $\nu_0$ has a density $V_0\in L^2(0,\infty)$, then $	\Vert V_t \Vert_2 \leq \Vert V_0 \Vert_2$ for every $t \geq 0$.
\end{lem}
\begin{proof}
	Arguing as in the proof of \cite[Prop.~2.1]{hambly_ledger_sojmark_2018} we get
	\[
	V_t(x) \leq \int_0^\infty \frac{1}{\sqrt{2\pi t}} e^{-\frac{(x-x_0+\alpha L_t)^2}{2t}}V_0(x_0)dx_0
	\leq \Big( \int_0^\infty \frac{1}{\sqrt{2\pi t}} e^{-\frac{(x-x_0+\alpha L_t)^2}{2t}}V_0(x_0)^2 dx_0 \Big)^{1/2},
	\]
	where the second inequality follows from Cauchy--Schwarz. Squaring and integrating with respect to $x$ gives the result.
\end{proof}

\subsubsection{Proof of Proposition \ref{prop:analyticity}}

Let $\langle \nu_t, \phi \rangle :=\int \!\phi(x) \nu_t(dx) = \mathbb{E}[\phi(X_t)\mathbf{1}_{t<\tau}]$ and, for the rest of this section, let $t_\star$ be the first jump time of $L$. Be definition, $X$ is continuous strictly before time $t_\star$. Thus, applying It\^o's formula to the stopped process $X_{\cdot\land\tau}$, we obtain the weak PDE
\[
\langle \nu_{t}, \phi \rangle 
= \langle \nu_0 , \phi \rangle
+ \frac{1}{2} \int^t_0 \langle \nu_s , \phi^{\prime\prime} \rangle ds 
- \alpha \int^{t}_0 \langle \nu_s ,  \phi^{\prime} \rangle dL_s, \;\; t\in[0,t_\star)
\]
for test functions $\phi \in C^2(0,\infty)$ with $\phi(0)=0$, where we have used that $\phi(0)=0$ implies $\phi(X_{t\land\tau})=\mathbf{1}_{t<\tau}\phi(X_t)$ for $t<t_\star$. Now take $\phi := T_\delta \psi$ for an arbitrary $\psi \in C^\infty_0(0,\infty)$. Then $\phi \in C^\infty_0(0,\infty)$, so integrating by parts and differentiating $n$ times, we can deduce that
\[
d \partial_x^n T_\delta \nu_t(x) = \tfrac{1}{2} \partial_x^{n+2} T_\delta \nu_t(x) dt + \alpha \partial_x^{n+1} T^r_\delta \nu_t(x) dL_t.
\]
for $x \geq 0$ (a.e.). Using that $d(\partial_x^n T_\delta \nu_t)^2=2\partial_x^n T_\delta \nu_t(x)d \partial_x^n T_\delta \nu_t(x)$, and rearranging, we get
\begin{align}
\label{eq:Analytic_Square}
d \bigl(\partial_x^n T_\delta \nu_t(x)\bigr)^2 
&= \partial_x^n T_\delta \nu_t(x)  \partial_x^{n+2} T_\delta \nu_t(x) dt 
+ 2\alpha  \partial_x^n T_\delta \nu_t(x) \partial_x^{n+1} T_\delta \nu_t(x) dL_t   \\
&\qquad + 4\alpha  \partial_x^n T_\delta \nu_t(x) \partial_x^{n+1} R_\delta  \nu_t(x) dL_t \nonumber ,
\end{align}
where we have introduced the remainder term
\[
R_\delta \nu_t(x) := \int_0^\infty \frac{1}{\sqrt{2\pi \delta} } e^{-\frac{(x+x_0)^2}{2 \delta}} \nu_t(dx_0). 
\]


The point of \eqref{eq:Analytic_Square} is of course to obtain $L^2$ estimates for the derivatives of $T_\delta \nu_{t}$. Since we are only interested in interior regularity, and hence only need local estimates, we can rely on cut-off functions for our arguments. To this end, we start by fixing any two open sets $U\Subset W\Subset(0,\infty)$, where `$\Subset$' denotes compact containment.  Then we let $\zeta$ be a smooth cut-off function with $\zeta=1$ on $U$,  $\zeta \in (0,1)$ on $W\setminus U$, and $\zeta=0$ otherwise. Note that $|\partial_x\zeta|+|\partial_{xx}\zeta| \leq C\mathbf{1}_{W\setminus U}$, where $C$ only depends on $W$ and $U$.

\begin{prop}[Smoothed energy estimate]
	\label{Analytic_Lem_Inductive_I}
	For all integers $a\geq2$ and $b\geq1$ we have
	\begin{align*}
	t_\star^b &\Vert \zeta^{\tfrac{a}{2}} \partial_x^n T_\delta \nu_{t_\star-} \Vert_2^2
	+\int^{t_\star}_0 t^b \Vert  \zeta^{\tfrac{a}{2}}\partial_x^{n+1} T_\delta \nu_t  \Vert_2^2 dt \\
	&\quad\leq c_1 a(a-1) \int^{t_\star}_0 t^b \Vert \zeta^{\tfrac{a-2}{2}} \partial_{x}^{n} T_\delta \nu_t \Vert_2^2 dt 
	+ b\int^{t_\star}_0 t^{b-1}  \Vert  \zeta^{\tfrac{a}{2}}\partial_x^{n} T_\delta \nu_t  \Vert_2^2 dt + o(1),
	\end{align*}
	as $\delta \to 0$, where $t_\star > 0$ the first jump time of $L$.
\end{prop}

\begin{proof}
	Multiplying (\ref{eq:Analytic_Square}) by $\zeta^a$ and using the integration-by-parts formulae
	\[
	\int g \cdot f \cdot f' = - \int g |f|^2
	\quad \text{and} \quad\int g \cdot f \cdot f'' =  -\int g |f'|^2 + \frac{1}{2}\int g'' |f|^2,
	\]
	we get
	\begin{align*}
	d\Vert \zeta^{\tfrac{a}{2}} \partial_x^n T_\delta \nu_t \Vert_2^2
	+ \Vert  \zeta^{\tfrac{a}{2}}\partial_x^{n+1} T_\delta \nu_t  \Vert_2^2 & dt 
	=\tfrac{1}{2} \Vert |\partial_{xx} \zeta^a|^{\tfrac{1}{2}} \partial_x^n T_\delta \nu_t \Vert_2^2 dt 
	- 2\alpha \Vert \zeta^{\tfrac{a}{2}} \partial_x^n T_\delta \nu_t \Vert_2^2 dL_t \\
	&+4\alpha \Big( \int^\infty_0 \zeta^a(x) \partial_x^{n} T_\delta \nu_{t}(x) \partial_x^{n+1} R_\delta \nu_t(x) dx\Big) dL_t.
	\end{align*}
	Since $|\partial_x\zeta|, |\partial_{xx}\zeta| \leq C$, Young's inequality with small enough parameter  $\eta > 0$ gives
	\begin{align*}
	d\Vert \zeta^{\tfrac{a}{2}} \partial_x^n T_\delta \nu_t \Vert_2^2
	&+ \Vert  \zeta^{\tfrac{a}{2}}\partial_x^{n+1} T_\delta \nu_t  \Vert_2^2 dt \\
	&\leq c_1 a(a-1) \Vert \zeta^{\tfrac{a-2}{2}} \partial_{x}^{n} T_\delta \nu_t \Vert_2^2 dt 
	+ 4\alpha \eta^{-1} \Vert \zeta^{\tfrac{a}{2}} \partial_x^{n+1} R_\delta \nu_t \Vert_2^2 dL_t,
	\end{align*}
	where we have used $a\zeta^{(a-1)/2} \leq a(a-1) \zeta^{(a-2)/2}$.
	Differentiating $t\mapsto t^b \Vert \zeta^{\tfrac{a}{2}} \partial_x^n T_\delta \nu_t \Vert_2^2$ with respect to $t$ and taking the range of integration up to $t_\star$ gives
	\begin{align*}
	t_\star^b \Vert& \zeta^{\tfrac{a}{2}} \partial_x^n T_\delta \nu_{t_\star-} \Vert_2^2
	+\int^{t_\star}_0 t^b \Vert  \zeta^{\tfrac{a}{2}}\partial_x^{n+1} T_\delta \nu_t  \Vert_2^2 dt 
	+ b\int^{t_\star}_0 t^{b-1}  \Vert  \zeta^{\tfrac{a}{2}}\partial_x^{n} T_\delta \nu_t  \Vert_2^2 dt \\
	&\leq c_1 a(a-1) \int^{t_\star}_0 t^b \Vert \zeta^{\tfrac{a-2}{2}} \partial_{x}^{n} T_\delta \nu_t \Vert_2^2 dt 
	 + 4\alpha \eta^{-1} \int^{t_\star}_0 t^b \Vert \zeta^{\tfrac{a}{2}} \partial_x^{n+1} R_\delta \nu_t \Vert_2^2 dL_t. 
	\end{align*}

	It remains to show the final term above vanishes as $\delta \to 0$. Writing $p_\delta(\cdot)$ for the standard Gaussian transition kernel, it follows from the definition of $R_\delta$ that
	\begin{align*}
	|\partial_x^{n+1} R_\delta \nu_t(x)|
	\leq \Big| \partial_x^{n+1}\! \int_0^\infty \! p_\delta(x+y) \nu_t(dy) \Big| \leq \int^\infty_0 |Q(\delta^{-\frac{1}{2}}y, \delta^{-\frac{1}{2}}x)|p_\delta(x+y) \nu_t(dy),
	\end{align*} 
	for a two-variable polynomial $Q$. Note $p_\delta(x+y) \leq p_\delta(y)\cdot e^{-x^2 / 2\delta}$. Applying Lemma \ref{Kernel_Lem_L2density} and Cauchy--Schwarz,
	we conclude $|\partial_x^{n+1} R_\delta \nu_t(x)| = O(e^{-x^2 / 4\delta})$ uniformly in $t$. Since $\zeta$ is supported on $W$,
	\[
	\Vert \zeta^{\tfrac{a}{2}} \partial_x^{n+1} R_\delta \nu_t \Vert_2 
	= O(e^{-w^2/2\delta})
	\]
	uniformly in $t$, where $w := \inf W>0$.  This is sufficient to complete the proof. 
\end{proof}

As Proposition \ref{Analytic_Lem_Inductive_I} gives control over the $(n+1)^\textrm{th}$ spatial derivative in terms of the $n^\textrm{th}$ derivative, we can use it to inductively prove that $V_{t_\star-}$ has weak derivatives of all orders, and is therefore smooth in the interior of $(0,\infty)$.

\begin{cor}[Smoothness]
	\label{Analytic_Cor_Smoothness}
	If $\nu_0$ has a density $V_0 \in L^2(0,\infty)$, then $V_{t_\star-} \in C^\infty(0,\infty)$. Furthermore $V_t \in C^\infty(0,\infty)$ for almost all $t \in (0,t_\star)$ and 
	\begin{align*}
	t_\star^b \Vert \zeta^{\tfrac{a}{2}} \partial_x^n V_{t_\star-} \Vert_2^2
	&+\int^{t_\star}_0 t^b \Vert  \zeta^{\tfrac{a}{2}}\partial_x^{n+1} V_t  \Vert_2^2 dt \\
	&\leq c_1 a(a-1) \int^{t_\star}_0 t^b \Vert \zeta^{\tfrac{a-2}{2}} \partial_{x}^{n} V_t \Vert_2^2 dt 
	+ b\int^{t_\star}_0 t^{b-1}  \Vert  \zeta^{\tfrac{a}{2}}\partial_x^{n} V_t  \Vert_2^2 dt.
	\end{align*}
\end{cor}

\begin{proof}
	Fix $n \geq 0$. Suppose that for all $U \Subset W \Subset (0,\infty)$, $a \geq 2n$, and $b \geq n$ we have
	\begin{equation}
	\label{eq:Analytic_Ind_I}
	\liminf_{\delta \to 0} \int^{t_\star}_0 t^b \Vert \zeta^{\tfrac{a}{2}} \partial_x^n T_\delta \nu_t \Vert_2^2 dt 
	< \infty.
	\end{equation}
	Then Fatou's Lemma and Lemma \ref{Kernel_Lem_MainLemma} imply that $\partial_x^n V_t$ exists and is in $L^2(U)$ for every $U \Subset (0,\infty)$ and almost all $t \in (0,t_\star)$, with
	\[
	\int^{t_\star}_0 t^b \Vert \zeta^{\tfrac{a}{2}} \partial_x^n T_\delta \nu_t \Vert_2^2 dt
	\to  \int^{t_\star}_0 t^b \Vert \zeta^{\tfrac{a}{2}} \partial_x^n V_t \Vert_2^2 dt < \infty.
	\]
	Therefore taking a $\liminf$ over $\delta$ in Proposition \ref{Analytic_Lem_Inductive_I}  gives that (\ref{eq:Analytic_Ind_I}) holds for $n+1$, for any $a \geq 2(n+1)$, $b \geq n+1$. Since (\ref{eq:Analytic_Ind_I}) holds for $n=0$ and all $a, b \geq 0$, by Lemma \ref{lem:L2_initial_bound}, using induction we conclude that the statement holds for all $n \geq 0$. Returning to Lemma \ref{Analytic_Lem_Inductive_I} once more, we deduce that
	\[
	\liminf_{\delta \to 0} \Vert  \partial_x^n T_\delta \nu_{t_\star-}\Vert_{L^2(U)}^2 \leq   \liminf_{\delta \to 0} \Vert \zeta^n \partial_x^n T_\delta \nu_{t_\star-} \Vert_2^2 < \infty,
	\]
	for every $n \geq 0$. Hence $V_{t_\star-}$ has weak derivatives of all orders at any point $x\in U$, and so it is smooth on $(0,\infty)$, since $U$ was arbitrary. Finally, sending $\delta \to 0$ in Lemma \ref{Analytic_Lem_Inductive_I} gives the estimate from the statement. 
\end{proof}

Finally we are in a position to complete the proof of Proposition \ref{prop:analyticity}.

\begin{proof}[Proof of Proposition \ref{prop:analyticity}] It suffices to prove the result at the first jump time, which we denote by $t_\star$ as in the above.
	Introduce the short-hand notation
	\[
	I(n,a,b) := \int^{t_\star}_0 t^b \Vert \zeta^{\tfrac{a}{2}} \partial_x^n V_t \Vert_2^2 dt,
	\]
	where we recall the cut-off $\zeta$ is defined for fixed $U \Subset W \Subset (0,\infty)$. Corollary \ref{Analytic_Cor_Smoothness} implies
	\[
	I(n+1, a, b) 
	\leq c_1 a (a-1) I(n, a-2,b) + bI(n,a,b-1)
	\leq (c_1 t_\star a(a-1) + b) I(n, a-2, b-1). 
	\]
	Iterating the argument gives
	\[
	I(n,2n,n) \leq I(0,0,0) \prod_{1 \leq i \leq n}(c_1 t_\star 2i(2i-1) + i) \leq C^n \cdot (2n)!,
	\]
	for $C>0$ a constant depending on $V_0$ and $\zeta$. Returning to Corollary \ref{Analytic_Cor_Smoothness} we have 
	\[
	t_\star^b \Vert \zeta^{\tfrac{a}{2}} \partial_x^n V_{t_\star-} \Vert_2^2 
	\leq (c_1 t_\star a(a-1) + b)I(n, a-2, b-1).
	\]
	Therefore setting $a = 2n +2$ and $b = n+1$ gives
	\[
	\Vert \partial_x^n V_{t_\star-} \Vert_{L^2(U)}^2
	\leq t_\star^{-(n+1)} (c_1 t_\star (2n+2)(2n+1) + n+1) \cdot C^n \cdot (2n)! = (C')^n \cdot (2n)!,
	\]
	for a further constant $C' > 0$ also depending on $t_\star$. Since $(2n)! \leq 4^n  (n!)^2$, we conclude that 
	\begin{equation}
	\label{eq:Analytic_mainproof_I}
	\Vert \partial_x^n V_{t_\star-} \Vert_{L^2(U)}
	\leq (C'')^n \cdot n!,
	\qquad \textrm{for every } n \geq 0.
	\end{equation}

	Morrey's inequality \cite[Sect.~5.6, Thm.~4]{evans_2010} gives a constant $c_2 > 0$ such that
	\[
	\Vert \partial_x^n V_{t_\star-} \Vert_{L^\infty(U)} 
	\leq c_2 \Vert  \partial_x^n V_{t_\star-}  \Vert_{H^1(U)}
	\leq (C''')^n \cdot n!,
	\]
	where we have applied (\ref{eq:Analytic_mainproof_I}), for $C'''>0$ a further constant independent of $n$. This inequality guarantees that $x\mapsto V_{t_\star-}(x)$ is analytic in the interior of $U$, and therefore, since $U \Subset (0,\infty)$ was arbitrary, we conclude that $x\mapsto V_{t_\star-}(x)$ is analytic at every $x\in(0,\infty)$. 
	\end{proof}

\paragraph*{Acknowledgements.}

We are grateful for the comments and suggestions from two anonymous referees which have improved the presentation of the paper.

\bibliographystyle{alpha}

\end{document}